\documentclass[a4paper,10pt]{amsart}

\usepackage{blindtext}
\usepackage{quoting}
\usepackage{amsthm}
\usepackage[utf8]{inputenc}
\usepackage{verbatim}
\usepackage{amsmath}
\usepackage{amsfonts}
\usepackage{stackrel}
\usepackage{amssymb}
\usepackage{graphicx}
\usepackage[svgnames]{xcolor}
\usepackage{lipsum}
\usepackage{svg}
\usepackage{bbm}
\usepackage{caption}
\usepackage{tcolorbox}
\usepackage{epigraph}
\usepackage{ebproof}
\usepackage[cal=boondoxo]{mathalfa}
\usepackage[all,cmtip]{xy}
\usepackage{mathtools}
\usepackage{color}
\usepackage{mathrsfs}
\usepackage{subfig}

\usepackage{framed}
\usepackage[colorlinks]{hyperref}
\usepackage{tikz}
\usepackage{tikz-cd}
\usepackage{multirow}
\usepackage[percent]{overpic}
 \usepackage{quiver}

\hypersetup{
    colorlinks = true,
    linkbordercolor = {red},
   	linkcolor ={teal},
	anchorcolor = {pink},
	citecolor =  {orange},
	filecolor = {teal},
	menucolor = {teal},
	runcolor =  {teal},
	urlcolor = {teal},
}

\usepackage{cleveref}

\usepackage[all,cmtip]{xy}

\usetikzlibrary{calc}
\usetikzlibrary{matrix,arrows}
\usepackage{tikz-cd}
\usepackage[square, sort, numbers]{natbib}

\usepackage{calligra} 

    {\endMakeFramed}

    {\endMakeFramed}

\makeatletter
\g@addto@macro\bfseries{\boldmath}
\makeatother

\theoremstyle{definition}
\newtheorem{thm}{Theorem}[subsection]

\newtheorem*{thm*}{Theorem}

\newtheorem{con}[thm]{Construction}

\newtheorem{prop}[thm]{Proposition}
\newtheorem{defn}[thm]{Definition}
\newtheorem{cor}[thm]{Corollary}

\theoremstyle{remark}
\newtheorem{ach}[thm]{Achtung!}
\newtheorem*{structure*}{Structure}
\newtheorem{rem}[thm]{Remark}
\newtheorem{exa}[thm]{Example}
\newtheorem{notation}[thm]{Notation}

\newcommand\lan{\mathsf{lan}}

\newcommand\Set{\operatorname{\bf Set}}

\newcommand\colim{\operatorname{colim}}

\newcommand\ca{\mathcal {A}}
\newcommand\cb{\mathcal {B}}

\newcommand\cf{\mathcal {F}}

\newcommand\cg{\mathcal {G}}

\newcommand\ce{\mathcal {E}}

\newcommand\cl{\mathcal {L}}

\newcommand\ct{\mathcal {T}}
\newcommand\cp{\mathcal {P}}

\newcommand\cx{\mathcal {X}}
\newcommand\cy{\mathcal {Y}}

\DeclareFontFamily{U}{min}{}
\DeclareFontShape{U}{min}{m}{n}{<-> udmj30}{}

\title{Towards Higher Topology}
\author{Ivan Di Liberti$^\dag$}

\thanks{$^\dag$ This research was mostly developed during the PhD studies of the author and has been supported through the grant 19-00902S from the Grant Agency of the Czech Republic. The finalization of this research has been supported by the GACR project EXPRO 20-31529X and RVO: 67985840.}

\address{
\newline Ivan \textsc{Di Liberti}\newline
Institute of Mathematics\newline
Czech Academy of Sciences\newline
\v{Z}itn\'{a} 25, Prague, Czech Republic\newline
diliberti.math@gmail.com\newline
}

\begin{document}

\begin{abstract}
We categorify the adjunction between locales and topological spaces, this amounts to an adjunction between (generalized) bounded ionads and topoi. We show that the adjunction is idempotent. We relate this adjunction to the Scott adjunction, which was discussed from a more categorical point of view in \cite{thcat}. We hint that $0$-dimensional adjunction inhabits the categorified one.
\end{abstract}
\maketitle
\setcounter{tocdepth}{1}

{
  \hypersetup{linkcolor=black}
  \tableofcontents
}

\section*{Introduction}
The notion of \textit{geometric entity} changes over time and with the kind of problems we want to tackle. On a macroscopic/historical level, this unspoken debate can be seen as the evolution of what mathematicians intend for the concept of \textit{space}. In the realm of very general and abstract mathematics, topological spaces are probably the most widespread concept and need no introduction, while locales are the main concept in formal topology and constructive approaches to geometry \cite{Vickers2007}. Posets with directed suprema belong mostly to domain theory \cite{abramsky1994domain}, topological tools where used by Scott to describe their properties. While one cannot claim them to have a geometric nature per se, after Scott's treatment, it is not possible to discuss them ignoring their topological features. In a sense that will be clear in few lines, topological spaces, locales and posets with directed suprema represent the conceptual framing with which this paper is concerned.
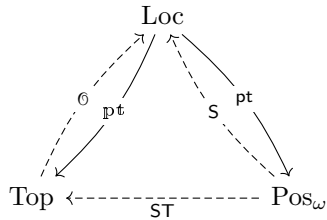
\begin{center}
\begin{tikzcd}
                                                                      & \text{Loc} \arrow[lddd, "\mathbbm{pt}" description, bend left=12] \arrow[rddd, "\mathsf{pt}" description, bend left=12] &                                                                                                                 \\
                                                                      &                                                                                                                  &                                                                                                                 \\
                                                                      &                                                                                                                  &                                                                                                                 \\
\text{Top} \arrow[ruuu, "\mathcal{O}" description, dashed, bend left=12] &                                                                                                                  & \text{Pos}_{\omega} \arrow[luuu, "\mathsf{S}" description, dashed, bend left=12] \arrow[ll, "\mathsf{ST}", dashed]
\end{tikzcd}
\end{center}
These three notions have been studied and related (as hinted by the diagram above) in the literature. Grothendieck introduced topoi as a generalization of locales of open sets of a topological space \cite{bourbaki2006theorie}. The geometric intuition has played a central rôle  since the very early days of this theory and has reached its highest peaks between the 70's and the 90's. Since its introduction, topos theory has gained more and more consensus and has shaped the language in which modern algebraic geometry has been written. The notion of ionad, on the other hand, was introduced by Garner much more recently \cite{ionads}. Whereas a topos is a categorified locale, a ionad is a categorified topological space. Accessible categories were independently defined by Lair and Rosický, even though their name was introduced by Makkai and Paré in \cite{Makkaipare}. 

The importance of accessible (and locally presentable) categories steadily grew and today they form an established framework to do category theory in the daily practice of the working mathematician. While ionads and topoi are arguably an approach to higher topology (whatever higher topology is), the significance of accessible categories with directed colimits is very far from being \textit{topological}, no matter which notion of the word we choose. In this paper we study the relationship between these two different approaches to higher topology and use this geometric intuition to analyze accessible categories with directed colimits.

We build on the analogy (and the existing literature on the topic) between \textit{plain} and higher topology and offer a generalization of the diagram above to \textbf{accessible categories with directed colimits, ionads and topoi}. 

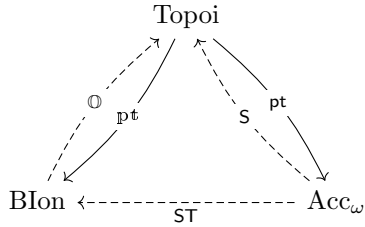
\begin{center}
\begin{tikzcd}
                                                                      & \text{Topoi} \arrow[lddd, "\mathbbm{pt}" description, bend left=12] \arrow[rddd, "\mathsf{pt}" description, bend left=12] &                                                                                                                 \\
                                                                      &                                                                                                                  &                                                                                                                 \\
                                                                      &                                                                                                                  &                                                                                                                 \\
\text{BIon} \arrow[ruuu, "\mathbb{O}" description, dashed, bend left=12] &                                                                                                                  & \text{Acc}_{\omega} \arrow[luuu, "\mathsf{S}" description, dashed, bend left=12] \arrow[ll, "\mathsf{ST}", dashed]
\end{tikzcd}
\end{center}

The functors that regulate the interaction between these different approaches to geometry are the main characters of the paper. These are the \textbf{Scott adjunction}, which relates accessible categories with directed colimits to topoi and the \textbf{categorified Isbell adjunction}, which relates bounded ionads to topoi. Since we have a quite good understanding of the topological case, we use this intuition to guess and infer the behavior of its categorification. In the diagram above, the only functor that has already appeared in the literature is $\mathbb{O}$, in \cite{ionads}. The adjunction $\mathsf{S} \dashv \mathsf{pt}$ was achieved in collaboration with Simon Henry and has appeared in \cite{simon} and \cite{thcat}.

\subsection*{Idempotency}
The adjunction between topological spaces and locales is idempotent, and it restricts to the well known equivalence of categories between locales with enough points and sober spaces. We obtain an analogous result for the adjunction between bounded ionads and topoi, inducing an equivalence between sober (bounded) ionads and topoi with enough points. Studying the Scott adjunction is more complicated, and in general it does not behave as well as the categorified Isbell duality does, yet we manage to describe the class of those topoi which are fixed by the Scott-comonad.

\subsection*{Relevance and Impact: Stone-like dualities $\&$ higher semantics}
 This topological picture fits the pattern of Stone-like dualities. As the latter is related to completeness results for propositional logic, its categorification is related to syntax-semantics dualities between categories of models and theories. Indeed a logical intuition on topoi and accessible categories has been available for many years \cite{sheavesingeometry}. A topos can be seen as a placeholder for a geometric theory, while an accessible category can be seen as a category of models of some infinitary theory. In a further work \cite{thlo} we introduce a new point of view on ionads, defining \textbf{ionads of models} of a geometric sketch. This approach allows us to entangle the theory of classifying topoi with the Scott and Isbell adjunctions providing several comparison results in this direction.

\begin{structure*}
The paper is divided in four sections,
\begin{enumerate}
  \item[\S \ref{topology}] we recall the concrete topology on which the analogy is built: the Isbell duality, relating topological spaces to locales. We also relate the Isbell duality to Scott's seminal work on the Scott topology, this first part is completely motivational and expository. This section contains the posetal version of \Cref{categorification} and \Cref{soberenough}.
  \item[\S \ref{backgroundionads}] We recall Garner's definition of ionad and we generalize it to large (but locally small) finitely pre-cocomplete categories. We provide some new technical tools.
  \item[\S \ref{categorification}]
  We introduce the higher dimensional analogs of topological spaces, locales and posets with directed colimits: ionads, topoi and accessible categories with directed colimits. We categorify the Isbell duality building on Garner's work on ionads, and we relate the Scott adjunction to its posetal analog.  \begin{tcolorbox}
    \begin{thm*}[\ref{categorified isbell adj thm}, Categorified Isbell adjunction]
    There is  a $2$-adjunction, $$ \mathbb{O}: \text{BIon} \leftrightarrows \text{Topoi}: \mathbbm{pt}. $$
    \end{thm*}
    \end{tcolorbox}
  The left adjoint of this adjunction was found by Garner; in order to find a right adjoint, we must allow for large ionads (\textbf{\Cref{fromtopoitobion}}). Finally, we show that the analogy on which the paper is built is deeply motivated and we show how to recover the content of the first section from the following ones (\textbf{\Cref{recovertopology}}).
  \item[\S \ref{soberenough}] We study the the categorified version of Isbell duality. This amounts to the notion of sober ionad and topos with enough points. Building on the \textbf{idempotency} of the categorified Isbell adjunction \textbf{(\Cref{idempotencyisbellcategorified})}, we derive properties of the Scott adjunction and we describe those topoi for which the counit of the Scott adjunction is an equivalence of categories (\textbf{\Cref{scottidempotency}}).
  \end{enumerate}
\end{structure*}

\subsection*{Notations and conventions} \label{backgroundnotations}

Most of the notation will be introduced when needed and we will try to make it as natural and intuitive as possible, but we would like to settle some notation.
\begin{enumerate}
\item $\ca, \cb$ will always be accessible categories, very possibly with directed colimits.
\item $\cx, \cy$ will always be ionads.
\item $\mathsf{Ind}_\lambda$ is the free completion under $\lambda$-directed colimits.
\item $\ca_{\kappa}$ is the full subcategory of $\kappa$-presentable objects of $\ca$.
\item $\cg, \ct, \cf, \ce$ will be Grothendieck topoi.
\item In general, $C$ is used to indicate small categories.
\item $\eta$ is the unit of the Scott adjunction.
\item $\epsilon$ is the counit of the Scott adjunction.
\item A Scott topos is a topos of the form $\mathsf{S}(\ca)$.
\item  $\P(X)$ is the category of small copresheaves of $X$.
\end{enumerate}

\begin{notation}[Presentation of a topos]\label{presentation of topos}
A presentation of a topos $\cg$ is the data of a  geometric embedding into a presheaf topos $f^*: \Set^{C} \leftrightarrows \cg :  f_*$. This means precisely that there is a suitable topology $\tau_f$ on $C$ that turns $\cg$ into the category of sheaves over $\tau$; in this sense $f$ \textit{presents} the topos as the category of sheaves over the site $(C, \tau_f)$. 
\end{notation}

\section{Spaces, locales and posets} \label{topology}

Our topological safari will start from the celebrated adjunction between locales and topological spaces. This adjunction is often attributed to Isbell, as \cite{10.2307/24490585} popularized this result to a broader audience, but one should point out that a \textit{fair} attribution of this result is essentially impossible to provide\footnote{An in-depth analysis of the historical evolution of the theory of locales can be found in \cite{Johnstone2001}.}. In \cite{10.2307/24490585}, Isbell credits the Paperts \cite{SE_1957-1958__1__A1_0}, Bénabou \cite{SE_1957-1958__1__A2_0} and Ehresmann \cite{Ehresmann1958GattungenVL}. In category theory, this attribution is especially unfortunate as the name \textit{Isbell duality} is already taken by the dual adjunction between presheaves and co-presheaves; this sometimes leads to some terminological confusion. The two adjunctions are similar in spirit, but do not generalize, at least not apparently, in any common framework. This first subsection is mainly expository and we encourage the interested reader to consult \cite{johnstone1986stone} and \cite[Chap. IX]{sheavesingeometry} for an extensive introduction. The aim of the subsection is not to introduce the reader to these results and objects, it is instead to organize them in a way that is useful to our purposes. More precise references will be given within the section.

\subsection{Spaces, Locales and Posets} \label{topologyintro}
This subsection tells the story of the diagram below. Let us bring every character to the stage.

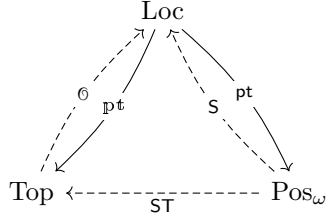
\begin{center}
\begin{tikzcd}
                                                                      & \text{Loc} \arrow[lddd, "\mathbbm{pt}" description, bend left=12] \arrow[rddd, "\mathsf{pt}" description, bend left=12] &                                                                                                                 \\
                                                                      &                                                                                                                  &                                                                                                                 \\
                                                                      &                                                                                                                  &                                                                                                                 \\
\text{Top} \arrow[ruuu, "\mathcal{O}" description, dashed, bend left=12] &                                                                                                                  & \text{Pos}_{\omega} \arrow[luuu, "\mathsf{S}" description, dashed, bend left=12] \arrow[ll, "\mathsf{ST}", dashed]
\end{tikzcd}
\end{center}


\begin{enumerate}
\item[]
\item[Loc] is the category of locales. It is defined to be the opposite category of frames, where objects are frames and morphisms are morphisms of frames.
\item[Top] is the category of topological spaces and continuous mappings between them.
\item[$\text{Pos}_{\omega}$] is the category of posets with directed suprema and functions preserving directed suprema.
\item[$\mathcal{O}$] associates to a topological space its frame of open sets and to each continuous function its inverse image.
\item[$\mathsf{ST}$] equips a poset with directed suprema with the Scott topology \cite[Chap. II, 1.9]{johnstone1986stone}.  This functor is fully faithful, i.e. a function is continuous with respect to the Scott topologies if and only if preserves suprema.
\item[$\mathsf{S}$] is the composition of the previous two; in particular the diagram above commutes.
\end{enumerate}



Both the functors $\mathcal{O}$ and $\mathsf{S}$ have a right adjoint; we indicate them by $\mathsf{pt}$ and $\mathbbm{pt}$, which in both cases stands for \textit{points}. The forthcoming discussion will bring out the reason for this clash of names; indeed the two functors look alike. The adjunction on the left (of the diagram above) is \textit{Isbell duality} (see \citep[IX.1-3]{sheavesingeometry}), while the one on the right was not named yet to our knowledge and we will refer to it as the \textit{(posetal) Scott adjunction}. There exists a natural transformation $$\iota: \mathsf{ST} \circ \mathsf{pt} \Rightarrow \mathbbm{pt},$$ which will be completely evident from the description of the functors. 


\begin{defn}[The (posetal) Scott adjunction] \label{scottposetal}
To our knowledge, the adjunction $\mathsf{S} \dashv \mathsf{pt}$ does not appear explicitly in the literature. Let us give a description of both the functors involved.
\begin{itemize}
  \item[$\mathsf{pt}$] The underlying set of $\mathsf{pt}(L)$ is the same as for Isbell duality. Its posetal structure is inherited from $\mathbb{T} = \{0 < 1\}$; in fact $\mathsf{pt}(L) = \text{Frm}(L,\mathbb{T})$ has a natural poset structure with directed unions given by pointwise evaluation \cite[1.11]{Vickers2007}.
  \item[$\mathsf{S}$] Given a poset $P$, its Scott locale $\mathsf{S}(P)$ is defined by the frame $\text{Pos}_{\omega}(P, \mathbb{T})$. It's quite easy to show that this poset is a locale.
  \end{itemize}
Observe that also this adjunction is a dual one, and is induced by precisely the same object as for Isbell duality. 
\end{defn}



\subsection{Sober spaces and spatial locales}\label{topologysoberandspatial}

As it is well known, it turns out that in full generality one cannot recover a space from the formal points of its frame of opens. Let us briefly recall the main results on this topic.

\begin{rem}[Unit and counit]\label{topologyisbellunitcounit}
Given a space $X$ the unit of the Isbell adjunction $\eta_X: X \to (\mathsf{pt} \circ \mathcal{O})(X)$ might not be injective. This is due to the fact that if two points $x,y$ in $X$ are such that $\mathsf{cl}(x) = \mathsf{cl}(y)$, then $\eta_X$ will confuse them.
\end{rem}


\begin{rem}[Idempotency]\label{topologyidempotency}
The Isbell adjunction $\mathcal{O} \dashv \mathsf{pt}$ is idempotent; this is proved in \citep[IX.3: Prop. 2, Prop. 3 and Cor. 4.]{sheavesingeometry}. It might look like this result has no qualitative meaning. Instead it means that given a locale $L$, the locale of opens sets of its points $\mathcal{O}\mathsf{pt}(L)$ is the best approximation of $L$ among spatial locales, namely those that are the locale of opens of a space. 

\end{rem}

\begin{thm}[{\citep[IX.3.3]{sheavesingeometry}}] The following are equivalent:
\begin{enumerate}
  \item $L$ has enough points;
  \item the counit $\epsilon_L : (\mathcal{O} \circ \mathsf{pt})(L) \to L$ is an isomorphism of locales;
  \item $L$ is the locale of open sets $\mathcal{O}(X)$ of some topological space $X$.
\end{enumerate}
\end{thm}

%


\subsection{From Isbell to Scott: topology}\label{topologyscottfromisbell}

\begin{rem}[Relating the Scott construction to the Isbell duality]
We anticipated that there exists a natural transformation $\iota: \mathsf{ST} \circ \mathsf{pt} \Rightarrow \mathbbm{pt}$.

\begin{center}
\begin{tikzcd}
           & \text{Loc} \arrow[lddd, "\mathbbm{pt}" description] \arrow[rddd, "\mathsf{pt}"{name=pt}, description] &                                               \\
           &                                                                                             &                                               \\
           &                                                                                             &                                               \\
\text{Top} \ar[Rightarrow, from=pt, shorten >=1pc, "\iota", shorten <=1pc]&                                                                                             & \text{Pos}_{\omega} \arrow[ll, "\mathsf{ST}"]
\end{tikzcd}
\end{center}
The natural transformation is pointwise given by the identity (for the underlying set is indeed the same), and witnessed by the fact that every Isbell-open is a Scott-open (\Cref{scottposetal}). This observation is implicit in \cite[II, 1.8]{johnstone1986stone}.
\end{rem}

\begin{rem}[Scott is not always sober] \label{topologyscottnotsober}
In principle $\iota$ might be an isomorphism. Unfortunately it was shown by Johnstone in \cite{10.1007/BFb0089911} that some Scott-spaces are not sober. Since every space in the image of $\mathbbm{pt}$ is sober, $\iota$ cannot be an isomorphism at least in those cases.
\end{rem}

\begin{rem}[Scott from Isbell] \label{topologyscottfromisbellthm}
Let us conclude with a version of {\citep[IX.3.3]{sheavesingeometry}} for the Scott adjunction. This has guided us in understanding the correct idempotency of the Scott adjunction.
\begin{thm}[Consequence of {\citep[IX.3.3]{sheavesingeometry}}] The following are equivalent:
\begin{enumerate}
  \item $L$ has enough points and $\iota$ is an isomorphism at $L$;
  \item the counit of the Scott adjunction is an isomorphism of locales at $L$.
\end{enumerate}
\end{thm}
\begin{proof}
It follows directly from {\citep[IX.3.3]{sheavesingeometry}} and the fact that $\mathsf{ST}$ is fully faithful.
\end{proof}
\end{rem}

\section{Revisiting Ionads} \label{backgroundionads}

Ionads were defined by Garner in \cite{ionads}, and to our knowledge that's all the literature available on the topic. His definition is designed to generalize the definition of topological space. Indeed a topological space $\cx$ is the data of a set (of points) and an interior operator, $$\text{Int}: 2^X \to 2^X.$$ Garner builds on the well known analogy between powerset and presheaf categories and extends the notion of interior operator to a presheaf category. The whole theory is extremely consistent with the expectations: while the poset of (co)algebras for the interior operator is the locale of open sets of a topological space, the category of coalgebras of a ionad is a topos, a natural categorification of the concept of locale.

\subsection{Garner's definitions}
\begin{defn}[Ionad]
An ionad $\cx = (X, \text{Int})$ is a set $X$ together with a comonad $\text{Int}: \Set^X \to \Set^X$ preserving finite limits.
\end{defn}

\begin{defn}[Category of opens of a ionad]
The category of opens $\mathbb{O}(\cx)$ of a ionad $\cx = (X, \text{Int})$ is the category of coalgebras of $\text{Int}$. We shall denote by $U_{\cx}$ the forgetful functor $U_\cx: \mathbb{O}(\cx) \to \Set^X$.  
\end{defn}

\begin{defn}[Morphism of Ionads]
A morphism of ionads $f: \cx \to \cy$ is a couple $(f, f^\sharp)$ where $f: X \to Y$ is a set function and $f^\sharp$ is a lift of $f^*$,
\begin{center}
\begin{tikzcd}
\mathbb{O}(\cy) \arrow[rr, "f^\sharp" description] \arrow[dd, "U_\cy" description] &  & \mathbb{O}(\cx) \arrow[dd, "U_\cx" description] \\
                                                                               &  &                                                  \\
\Set^Y \arrow[rr, "f^*" description]                                           &  & \Set^X                                          
\end{tikzcd}
\end{center}
\end{defn}

\begin{defn}[Specialization of morphism of ionads]
Given two morphism of ionads $f,g: \cx \to \cy$, a specialization of morphism of ionads $\alpha: f \Rightarrow g$ is a natural transformation between $f^\sharp$ and $g^\sharp$,
\begin{center}
\begin{tikzcd}
\mathbb{O}(\cy) \arrow[r, "f^\sharp" description, bend left=35, ""{name=U, below}]
\arrow[r,"g^\sharp" description, bend right=35, ""{name=D}]
& \mathbb{O}(\cx)
\arrow[Rightarrow, "\alpha" description, from=U, to=D] \end{tikzcd}
\end{center}
\end{defn}

\begin{defn}[$2$-category of Ionads]
The $2$-category of ionads has ionads as objects, morphism of ionads as $1$-cells and specializations as $2$-cells.
\end{defn}

\begin{defn}[Bounded Ionads]
A ionad $\cx$ is bounded if $\mathbb{O}(\cx)$ is a topos.
\end{defn}

\subsection{A generalization of Garner's ionads}
In his paper, Garner mentions that in giving the definition of ionad he could have chosen a category instead of a set \cite[Rem. 2.4]{ionads}. For technical reasons we cannot be content with Garner's original definition, thus we will allow ionads over a category (as opposed to sets), even a locally small (but possibly large) one. We will need this definition later in the text to establish a connection between ionads and topoi (and thus categorify Isbell adjunction).


\begin{defn}[Pre-(co)limit {\cite[Def. 3.3]{presheaves}}]
By a pre-limit of a diagram $D$ is meant a set of cones such that every cone of $D$ factorizes through some of them (not uniquely in general). A category is pre-complete if every diagram has a pre-limit. Dual concept: pre-colimit and pre-cocomplete category.
\end{defn}

\begin{defn}[Generalized Ionads]
A generalized ionad $\cx = (X, \text{Int})$ is a locally small (but possibly large) pre-finitely cocomplete category $X$ together with a lex comonad $\text{Int}: \P(X) \to \P(X)$.
\end{defn}

\begin{ach}
We will always omit the adjective \textit{generalized}.
\end{ach}

We are well aware that the notion of generalized ionad seems a bit puzzling at first sight. \textit{Why isn't it just the data of a locally small category $X$ together with a lex comonad on $\Set^X$?} The answer to this question is a bit delicate, having both a technical and a conceptual aspect. In a nutshell, $\P(X)$ is a well-behaved full subcategory of $\Set^X$, while the existence of finite pre-colimits will ensure us that $\P(X)$ has finite limits. Let us dedicate some remarks to make these hints more precise.

\begin{defn}[On small (co)presheaves]
By $\P(X)$ we mean the full subcategory of $\Set^X$  made by small copresheaves over $X$, namely those functors $X \to \Set$ that are small colimits of corepresentables (in $\Set^X$). 
\end{defn}

This is a locally small category, as opposed to $\Set^X$ which might be locally large. The study of small presheaves $X^\circ \to \Set$ over a category $X$ is quite important with respect to the topic of free completions under limits and under colimits. Obviously, when $X$ is small, every presheaf is small. Given a category $X$, its category of small presheaves is usually indicated by $\cp(X)$, while $\cp^\sharp(X)$ is $\cp(X^\circ)^\circ$. 

The most updated account on the property of $\cp(X)$ is given by \cite{presheaves} and \cite{DAY2007651}. $\cp(X)$ is the free completion of $X$ under colimits, while $\cp^\sharp(X)$ is the free completion of $X$ under limits. Thus, $\P(X)$ is the free completion of $X^\circ$ under colimits. The following equation clarifies the relationship between $\cp, \P$ and $\cp^\sharp$, \[\cp^\sharp(X)^\circ = \P(X) = \cp(X^\circ).\]

In full generality $\cp(X)$ is not complete, nor it has any limit whatsoever. Yet, under some smallness condition most of the relevant properties of $\cp(X)$ remain true. Below we recall a good example of this behavior, and we address the reader to \cite{presheaves} for a for complete account.

\begin{prop}[{\cite[Cor. 3.8]{presheaves}}] \label{smallpreshcomplete}
$\cp(X)$ is (finitely) complete if and only if $X$ is (finitely) pre-complete\footnote{See {\cite[Def. 3.3]{presheaves}}.}.
\end{prop}

\begin{cor}
If $X$ is finitely pre-cocomplete, then $\P(X)$ has finite limits.
\end{cor}

These two results are the reason for which we need our categories to be pre-finitely cocomplete, indeed we need to discuss lex functors on $\P(X)$. The notion of pre-(co)limits is due to Freyd, but we prefer to refer to \cite{presheaves} because the paper is more readable. A precise understanding of the notion of pre-cocomplete category is actually not needed for our purposes, the following sufficient condition will be more than enough through the paper.

\begin{cor}[{\cite[Exa. 3.5 (b) and (c)]{presheaves}}]
If $X$ is small or it is accessible, then $\P(X)$ is complete.
\end{cor}

What must be understood is that being pre-complete, or pre-cocomplete should not be seen as a completeness-like property, instead it is much more like a smallness assumption.

\begin{exa}[Ionads are generalized ionads]
It is obvious from the previous discussion that a ionad is a generalized ionad.
\end{exa}

\begin{rem}[Small copresheaves vs copresheaves]
When $X$ is a finitely pre-cocomplete category, $\P(X)$ is an infinitary pretopos and finite limits are \textit{nice} in the sense that they can be computed in $\Set^X$. Being an infinitary pretopos, together with being the free completion under (small) colimits makes the conceptual analogy between $\P(X)$ and $2^X$ nice and tight, but there is also a technical reason to prefer small copresheaves to copresheaves.
\end{rem}

\begin{prop} \label{aggiuntifacili}
If $f^*: \cg \to \P(X)$ is a cocontinuous functor from a cocomplete, co-well powered category with a generating set, then it has a right adjoint $f_*$.
\end{prop}
\begin{proof}
It follows from the Special Adjoint Functor Theorem.
\end{proof}

\begin{rem}
The result above allows to produce comonads on $\P(X)$ (just compose $f^*f_*$) and follows from the Adjoint Functor Theorem, and uses the fact that $\P(X)$ is locally small to stay in place \cite{DAY2007651}, as opposed to $\Set^X$ which would have generated size issues.
\end{rem}

\begin{ach}
$\P(X)$ is a (Grothendieck) topos if and only if $X$ is an essentially small category, thus in most of the examples of our interest $\P(X)$ will not be a Grothendieck topos. Yet, we feel free to use a part of the terminology from topos theory (geometric morphism, geometric surjection, geometric embedding), because it is an infinitary pretopos (and thus only lacks a generator to be a topos).
\end{ach}

\subsubsection{Morphisms of ionads}

We still need to solve one last size issue. We need to tweak the notion of morphism of ionad too. Indeed, Garner's definition mentions $\Set^X$ and we want our definition to be based only on the notion of $\P(X)$.

\begin{defn}[Morphism of (generalized) Ionads]
A morphism of (generalized) ionads $f: \cx \to \cy$ is a couple $(f, f^\sharp)$ where:
\begin{enumerate}
  \item $f: X \to Y$ is a functor,
  \item for all small copresheaves $g: Y \to \Set$, $f^*(g)$ is small,
  \item $f^\sharp$ is a lift of $f^*$,
  \begin{center}
\begin{tikzcd}
\mathbb{O}(\cy) \arrow[rr, "f^\sharp" description] \arrow[dd, "U_\cy" description] &  & \mathbb{O}(\cx) \arrow[dd, "U_\cx" description] \\
                                                                               &  &                                                  \\
\P(Y) \arrow[rr, "f^*" description]                                           &  & \P(X)                                          
\end{tikzcd}
\end{center}
\end{enumerate}
\end{defn}

Similarly to the discussion that we have provided for pre-(co)limits, we need a criterion to check that the condition (2) of this definition is automatically verified.

\begin{prop} \label{f*welldef}
Let $\ca \to \cb$ be an accessible functor between accessible categories, then $f^*$ maps small copresheaves to small copresheaves and thus the functor $f^*  : \P( \cb)\to \P( \ca)$ is well defined.
\end{prop}
\begin{proof}
Observe that for any accessible category $\ca$ we can give a different presentation for $\P(\ca)$ that is $\P(\ca) \simeq \text{Acc}(\ca, \Set)$, where by $\text{Acc}(\ca, \Set)$ we intend the category of all accessible functors \cite[Sec. 2]{DAY2007651}. Now let $g$ be a functor in $\text{Acc}(\cb, \Set)$. Since $f$ is accessible, $f^*(g) =g \circ f$ is accessible too\footnote{Accessible functors between accessible categories do compose.}, and thus lands in  $\text{Acc}(\ca, \Set)$, as desired.
\end{proof}

\subsection{Base of a ionad}

 In analogy with the notion of base for a topology, Garner defines the notion of base of a ionad \cite[Def. 3.1, Rem. 3.2]{ionads}. This notion will be a handy technical tool in the paper. Our definition is pretty much equivalent to Garner's one (up to the fact that we keep flexibility on the size of the base) and is designed to be easier to handle in our setting.

\begin{defn}[Base of a ionad]
Let $\cx = (X, \text{Int})$ be a ionad. We say that a flat functor $e: B \to \P(X)$ generates\footnote{This definition is just a bit different from Garner's original definition \cite{ionads}[Def. 3.1, Rem. 3.2]. We stress that in this definition, we allow for large basis.} the ionad if $\text{Int}$ is naturally isomorphic to the density comonad of $e$, \[\text{Int} \cong \lan_e e.\]
\end{defn}

\begin{exa}
The forgetful functor $U_\cx: \mathbb{O}(\cx) \to \P(X)$ is always a basis for the ionad $\cx$. This follows from the basic theory about density comonads: when $U_\cx$ is a left adjoint, its density comonad coincides with the comonad induced by its adjunction. This observation does not appear in \cite{ionads} because he only defined small bases, and it almost never happens that $\mathbb{O}(\cx)$ is a small category.
\end{exa}

In \cite[3.6, 3.7]{ionads}, the author lists three equivalent conditions for a ionad to be bounded. The conceptual one is obviously that the category of opens is a Grothendieck topos, while the other ones are more or less technical. In our treatment the equivalence between the three conditions would be false. But we have the following characterization.

\begin{prop}\label{critboundedionad}
A ionad $\cx = (X, \text{Int})$ is bounded if any of the following equivalent conditions is verified:
\begin{enumerate}
  \item $\mathbb{O}(\cx)$ is a topos.
  \item there exist a Grothendieck topos $\cg$ and a geometric surjection $f : \P(X) \twoheadrightarrow \cg$ such that $\text{Int} \cong f^*f_*$.
  \item there exist a Grothendieck topos $\cg$, a geometric surjection $f : \P(X) \twoheadrightarrow \cg$ and a flat functor $e: B \to \cg$ such that $f^*e$ generates the ionad.
\end{enumerate}
\end{prop}
\begin{proof}
Clearly $(1)$ implies $(2)$. For the implication $(2) \Rightarrow (3)$, it's enough to choose $e: B \to \cg$ to be the inclusion of any generator of $\cg$. Let us discuss the implication $(3) \Rightarrow (1)$. Let $\ce$ be the category of coalgebras for the density comonad of $e$ and call $g: \cg \to \ce$ the geometric surjection induced by the comonad, (in particular $\lan_ee \cong g^*g_*$). We claim that $\ce \simeq \mathbb{O}(\cx)$. Invoking \cite{sheavesingeometry}[VII.4 Prop. 4] and because geometric surjections compose, we have $\ce \simeq \mathsf{coAlg}(f^*g^*g_*f_*)$. The thesis follows from the observation that \[\text{Int} \cong \lan_{f^*e} (f^*e) \cong \lan_{f^*}( \lan_e (f^*e)) \cong \lan_{f^*}(f^* \lan_e e) \cong f^*g^*g_*f_*.\] 
\end{proof}

\subsection{Generating morphisms of ionads}
In the paper, we will need a practical way to induce morphism of ionads. The following proposition does not appear in \cite{ionads} and will be our main \textit{morphism generator}. From the perspective of developing technical tool in the theory of ionads, this proposition has an interest in its own right. 

The proposition below categorifies a basic lemma in general topology: let $f: X \to Y$ be a function between topological spaces, and let $B_X$ and $B_Y$ be bases for the respective topologies. If $f^{-1}(B_Y) \subset B_X$, then $f$ is continuous. Our original proof has been simplified by Richard Garner during the reviewing process of the author's Ph.D thesis.

\begin{prop}[Generator of morphism of ionads] \label{morphism of ionads generator}
Let $\cx$ and $\cy$ be ionads, respectively generated by bases $e_X: B \to \P(X)$ and $e_Y: C \to \P(Y)$. Let $f: X \to Y$ a functor admitting a lift as in the diagram below.
\begin{center}
\begin{tikzcd}
C \arrow[dd, "e_Y" description] \arrow[rr, "f^\diamond" description, dashed] &  & B \arrow[dd, "e_X" description] \\
                                                                             &  &                                 \\
\P(Y) \arrow[rr, "f^*" description]                                         &  & \P(X)                         
\end{tikzcd}
\end{center}
Then $f$ induces a morphism of ionads $(f, f^\sharp)$.
\end{prop}
\begin{proof}
By the discussion in \cite[Exa. 4.6, diagram (6)]{ionads}, it is enough to provide a morphism as described in the diagram below.
  \begin{center}
  \begin{tikzcd}
  C \arrow[dd, "e_Y" description] \arrow[rr, "f'" description, dashed] &  & \mathbb{O}(\lan_{e_X} e_X) \arrow[dd, "\mathsf{U}_\cx" description] \\
                                                                             &  &                                 \\
  \P(Y) \arrow[rr, "f^*" description]                                         &  & \P(X)                        
  \end{tikzcd}
  \end{center}
Also, \cite{ionads}[Exa. 4.6] shows that giving a map of ionads $\cx \to \cy$ is the same of giving $f: X \to Y$ and a lift of $C \to \P(Y) {\to} \P(X)$ through $\mathbb{O}(X)$. Applying this to the identity map $\cx \to \cx$ we get a lift of $B \to \P(X)$ trough $\mathbb{O}(\lan_{e_X} e_X)$.  Now composing that with $C \to B$ gives the desired square.
  \begin{center}
  \begin{tikzcd}
  C \arrow[dd, "e_Y" description] \arrow[rr, "f^\diamond" description] \arrow[rrrr, "f'" description, dashed, bend left=15] &  & B \arrow[rr] \arrow[dd, "e_X"] &  & \mathbb{O}(\lan_{e_X} e_X) \arrow[lldd, "\mathsf{U}_\cx" description] \\
                                                                                                                       &  &                                &  &                                                                       \\
  \P(Y) \arrow[rr, "f^*" description]                                                                                   &  & \P(X)                          &  &                                                                      
  \end{tikzcd}
  \end{center}
\end{proof}

\section{Ionads, Topoi and Accessible categories } \label{categorification}

\subsection{Setting the stage}
Now we come to the $2$-dimensional counterpart of the previous section. As in the previous one, this section is dedicated to describing the properties of a diagram.
\begin{center}
\begin{tikzcd}
                                                              & \text{Topoi} &                                                                                                 \\
                                                              &                                                                                                                  &                                                                                                 \\
                                                              &                                                                                                                  &                                                                                                 \\
\text{BIon} \arrow[ruuu, "\mathbb{O}" description, ] &                                                                                                                  & \text{Acc}_{\omega} \arrow[luuu, "\mathsf{S}" description, ] \arrow[ll, "\mathsf{ST}"]
\end{tikzcd}
\end{center}

\begin{enumerate}
\item[]
\item[Topoi] is the $2$-category of topoi and geometric morphisms.
\item[BIon] is the $2$-category of (generalized) bounded ionads.
\item[$\text{Acc}_{\omega}$] is the $2$-category of accessible categories \cite[Chap. 2]{adamekrosicky94} with directed colimits and functors preserving them.
\end{enumerate}


\begin{defn}[$\mathbb{O}$] \label{opensionad}
 Let us briefly recall the relevant definitions. A generalized bounded ionad $\cx = (X, \text{Int})$ is a (possibly large, locally small and finitely pre-cocomplete) category $X$ together with a comonad $\text{Int} : \P(X) \to \P(X)$ preserving finite limits whose category of coalgebras is a Grothendieck topos. $\mathbb{O}$ was described by Garner in \cite[Rem. 5.2]{ionads}, it maps a bounded ionad to its category of opens, that is the category of coalgebras for the interior operator.
\end{defn}

\begin{con}[$\mathsf{S}$] \label{defnS}

We recall the construction of $\mathsf{S}$ from \cite{simon} and \cite{thcat}. Let $\ca$ be an accessible category with directed colimits.  $\mathsf{S}(\ca)$ is defined as the category the category of functors preserving directed colimits into sets. \[\mathsf{S}(\ca) = \text{Acc}_{\omega}(\ca, \Set).\]
We will show in the next proposition that $\mathsf{S}(\ca)$ is a topos. Let $f: \ca \to \cb$ be a $1$-cell in $\text{Acc}_{\omega}$, the geometric morphism $\mathsf{S}f$ is defined by precomposition as described below.

\begin{center}
\begin{tikzcd}
\ca \arrow[dd, "f"] &  & \mathsf{S}\ca \arrow[dd, "f_*"{name=lower, description}, bend left=49] \\
                    &  &                                                \\
\cb                 &  & \mathsf{S}\cb \arrow[uu, "f^*"{name=upper, description}, bend left=49] 

 \ar[phantom, from=lower, to=upper, shorten >=1pc, "\dashv", shorten <=1pc]
\end{tikzcd}
\end{center}
$\mathsf{S}f = (f^* \dashv f_*)$ is defined as follows: $f^*$ is the precomposition functor $f^*(g) = g \circ f$. This is well defined because $f$ preserve directed colimits. $f^*$ is a functor preserving all colimits between locally presentable categories and thus has a right adjoint by the adjoint functor theorem, that we indicate with $f_*$. Observe that $f^*$ preserves finite limits because finite limits commute with directed colimits in $\Set$.
\end{con}

\begin{prop}\label{scottconstructionwelldefined}
 $\mathsf{S}(\ca)$ is a topos.
\end{prop}
\begin{proof}[Sketch of Proof]
By definition $\ca$ is $\lambda$-accessible for some $\lambda$. Obviously $\text{Acc}_\omega(\ca, \Set)$ sits inside $\lambda\text{-Acc}(\ca, \Set)$. Recall that $\lambda\text{-Acc}(\ca, \Set)$ is equivalent to $\Set^{\ca_\lambda}$ by the restriction-Kan extension paradigm and the universal property of $\mathsf{Ind}_\lambda$-completion.  This inclusion $i: \text{Acc}_\omega(\ca, \Set) \hookrightarrow \Set^{\ca_\lambda}$, preserves all colimits and finite limits, this is easy to show and depends on one hand on how colimits are computed in this category of functors, and on the other hand on the fact that in $\Set$ directed colimits commute with finite limits. By the adjoint functor theorem, $\text{Acc}_\omega(\ca, \Set)$ amounts to a coreflective subcategory of a topos whose associated comonad is left exact. By \cite[V.8 Thm.4]{sheavesingeometry}, it is a topos.
\end{proof}

\begin{prop}\label{coreflective} 
$\mathsf{S}(\ca)$ is coreflective in $\P(\ca)$.
 \end{prop}
 \begin{proof}
 There is an obvious inclusion of the Scott topos $\mathsf{S}(\ca)$ in the category of copresheaves $\Set^\ca$. As we have seen in the previous proposition, every functor preserving directed colimits has arity, and thus is small. It follows that the inclusion factors through the category of small copresheaves. The inclusion $i_\ca: \mathsf{S}(\ca) \to \P(\ca)$ has a right adjoint $r_\ca$, this follows from \Cref{scottconstructionwelldefined} and \Cref{aggiuntifacili}.
 \end{proof}

\begin{con}[$\mathsf{ST}$] \label{categorifiedscotttopology}
 The construction is based on \Cref{defnS} and \Cref{coreflective}, we map $\ca$ to the bounded ionad $(\ca, r_\ca i_\ca)$, as described in \Cref{coreflective}. A functor $f: \ca \to \cb$ is sent to the morphism of ionads $(f,f^\sharp)$  below, where $f^\sharp$ coincides with the inverse image of $\mathsf{S}(f)$.

 \begin{center}
\begin{tikzcd}
\mathsf{S}(\cb) \arrow[dd] \arrow[rr, "f^\sharp" description, dashed] &  & \mathsf{S}(\ca) \arrow[dd] \\
                                                                      &  &                            \\
\P(\cb) \arrow[rr, "f^*" description]                                &  & \P(\ca)                  
\end{tikzcd}
 \end{center}
Notice that is the diagram above $f^*$ is well defined because of \Cref{f*welldef}.  Coherently with the previous section, it is quite easy to notice that $\mathsf{S} \cong  \mathbb{O} \circ \mathsf{ST}$. Let us show it, \[  \mathbb{O} \circ \mathsf{ST}(\ca) = \mathbb{O} ( \mathsf{ST}(\ca)) \stackrel{\ref{coreflective}}{=} \mathsf{coAlg}(r_\ca i_\ca)   \simeq \mathsf{S}(\ca).  \]
\end{con}



\subsection{Points} \label{categorifiedpoints}

As in their analogs, both the functors $\mathbb{O}$ and $\mathsf{S}$ have a right (bi)adjoint. We indicate them by $\mathsf{pt}$ and $\mathbbm{pt}$, which in both cases stands for points.

\begin{center}
\begin{tikzcd}
                                                                      & \text{Topoi} \arrow[lddd, "\mathbbm{pt}" description, bend left=12] \arrow[rddd, "\mathsf{pt}" description, bend left=12] &                                                                                                                 \\
                                                                      &                                                                                                                  &                                                                                                                 \\
                                                                      &                                                                                                                  &                                                                                                                 \\
\text{BIon} \arrow[ruuu, "\mathbb{O}" description, dashed, bend left=12] &                                                                                                                  & \text{Acc}_{\omega} \arrow[luuu, "\mathsf{S}" description, dashed, bend left=12] \arrow[ll, "\mathsf{ST}", dashed]
\end{tikzcd}
\end{center}
This sub-subsection will be mostly dedicated to the construction of $\mathbbm{pt}$ and to show that it is a right (bi)adjoint for $\mathbb{O}$. Let us mention though that there exists a natural functor $$ \iota: \mathsf{ST} \circ \mathsf{pt} \Rightarrow \mathbbm{pt} $$ which is not in general an equivalence of categories.

\begin{defn}[The $2$-functor $\mathsf{pt}$]\label{pt}
The functor of points $\mathsf{pt}$ belongs to the literature since quite some time, $\mathsf{pt}$ is the covariant hom functor $\text{Topoi}(\Set, - )$. It maps a Grothendieck topos $\cg$ to its category of points, \[\cg \mapsto \text{Cocontlex}(\cg, \Set).\]
Of course given a geometric morphism $f: \cg \to \ce$, we get an induced morphism $\mathsf{pt}(f): \mathsf{pt}(\cg) \to \mathsf{pt}{(\ce)}$ mapping $p^* \mapsto p^* \circ f^*$. The fact that $\text{Topoi}(\Set, \cg)$ is an accessible category with directed colimits appears in the classical reference by Borceux as \citep[Cor. 4.3.2]{borceux_19943}, while the fact that $\mathsf{pt}(f)$ preserves directed colimits follows trivially from the definition.
\end{defn}

\begin{thm}[{\citep{simon}[Prop. 2.3], \citep{thcat}[Thm. 2.1]} The Scott adjunction]\label{scottadj}
The $2$-functor of points $\mathsf{pt} :\text{Topoi} \to \text{Acc}_{\omega} $ has a left biadjoint $\mathsf{S}$, yielding the Scott biadjunction, $$\mathsf{S} : \text{Acc}_{\omega} \leftrightarrows \text{Topoi}: \mathsf{pt}. $$
\end{thm}

\begin{con}[$\text{Topoi} \rightsquigarrow \text{Bion}$: every topos induces a generalized bounded ionad over its points]\label{categorified isbell construction of pt} \label{fromtopoitobion}
For a topos $\ce$, there exists a natural evaluation pairing $$\mathsf{ev}: \ce \times \mathsf{pt}(\ce) \to \Set,$$ mapping the couple $(e, p)$ to its evaluation $p^*(e)$. This construction preserves colimits and finite limits in the first coordinate, because $p^*$ is an inverse image functor. This means that its mate functor $ ev^*: \ce  \to \Set^{\mathsf{pt}(\ce)},$ preserves colimits and finite limits. Moreover $ev^*(e)$ preserves directed colimits for every $e \in \ce$. Indeed, \[ ev^*(e)(\colim p_i^*) \cong (\colim p_i^*)(e) \stackrel{(\star)}{\cong} \colim ((p_i^*)(e)) \cong  \colim (ev^*(e)(p_i^*)). \] where $(\star)$ is true because directed colimits of points are computed pointwise. Thus, since the category of points of a topos is always accessible (say $\lambda$-accessible) and $ev^*(e)$ preserves directed colimits, the value of $ev^*(e)$ is uniquely individuated by its restriction to $\mathsf{pt}(\ce)_\lambda$. Thus, $ev^*$ takes values in $\P(\mathsf{pt}(\ce))$.  By \Cref{aggiuntifacili}, $ev^*$ must have a right adjoint $ev_*$, and we get an adjunction, $$ev^* :\ce \leftrightarrows \P(\mathsf{pt}(\ce)): ev_*.$$ Since the left adjoint preserves finite limits, the comonad $ev^*ev_*$ is lex and thus induces a ionad over $\mathsf{pt}(\ce)$. This ionad is bounded, this follows from a careful analysis of the discussion above. Indeed $\Set^{\mathsf{pt}(\ce)_\lambda}$ is lex-coreflective in $\P(\mathsf{pt}(\ce))$ and  $S \to \ce \to \Set^{\mathsf{pt}(\ce)_\lambda} \leftrightarrows \P(\mathsf{pt}(\ce))$, where $S$ is a site of presentation of $\ce$, satisfies the hypotheses of  \Cref{critboundedionad}(3).
\end{con}

\begin{con}[$\text{Topoi} \rightsquigarrow \text{Bion}$: every geometric morphism induces a morphism of ionads]\label{categorified isbell construction of pt} 

Observe that given a geometric morphism $f: \ce \to \cf$, $\mathsf{pt}(f): \mathsf{pt}(\ce) \to \mathsf{pt}(\cf) $ induces a morphism of ionads $(\mathsf{pt}(f), \mathsf{pt}(f)^\sharp)$ between $\mathbbm{pt}(\ce)$ and $\mathbbm{pt}(\cf)$. In order to describe $\mathsf{pt}(f)^\sharp$, we invoke \Cref{morphism of ionads generator}[(a)]. Thus, it is enough to provide  a functor making the diagram below commutative (up to natural isomorphism). 
\begin{center}
\begin{tikzcd}
\cf \arrow[dd, "ev^*_{\cf}" description] \arrow[rr, "f^\diamond" description, dashed] &  & \ce \arrow[dd, "ev^*_{\ce}" description] \\
                                                                             &  &                                 \\
\P({\mathsf{pt}(\cf)}) \arrow[rr, "\mathsf{pt}(f)^*" description]                                         &  & \P({\mathsf{pt}(\ce))}                         
\end{tikzcd}
\end{center}
Indeed such a functor exists and coincides with the inverse image $f^*$ of the geometric morphism $f$. Notice that $\mathsf{pt}(f)^*$ is well defined because $\mathsf{pt}(f)$ is an accessible functor between accessible categories, thus we can apply \Cref{f*welldef}.
\end{con}

\begin{defn}[The $2$-functor $\mathbbm{pt}$] \label{categorified isbell defn of pt}
$\mathbbm{pt}(\ce)$ is defined to be the ionad $(\mathsf{pt}(\ce), ev^*ev_*)$, as described in the two previous constructions. 
\end{defn}

\begin{thm}[Categorified Isbell adjunction, $\mathbb{O} \dashv \mathbbm{pt} $]\label{categorified isbell adj thm}

$$ \mathbb{O}: \text{BIon} \leftrightarrows \text{Topoi}: \mathbbm{pt} $$
\end{thm}
\begin{proof}
We provide the unit and the counit of this adjunction. This means that we need to provide geometric morphisms $\rho: \mathbb{O}\mathbbm{pt}(\ce) \to \ce$ and morphisms of ionads $\lambda: \cx \to \mathbbm{pt}\mathbb{O} \cx $. Let's study the two problems separately.
\begin{itemize}
  \item[($\rho$)] As in the case of any geometric morphism, it is enough to provide the inverse image functor $\rho^*: \ce \to \mathbb{O}\mathbbm{pt}(\ce)$. Now, recall that the interior operator of $\mathbbm{pt}(\ce)$ is induced by the adjunction $ev^*: \ce \leftrightarrows \P({\mathsf{pt}(\ce)}): ev_*$ as described in the remark above. By the universal property of the category of coalgebras, the adjunction $\mathsf{U}: \mathbb{O}\mathbbm{pt}(\ce) \leftrightarrows \P({\mathsf{pt}}(\ce)): \mathsf{F}$ is terminal among those adjunctions that induce the comonad $ev^*ev_*$. This means that there exists a functor $\rho^*$ lifting $e^*$ along $\mathsf{U}$ as in the diagram below.

  \begin{center}
  \begin{tikzcd}
  \ce \arrow[rrdd, "ev^*" description] \arrow[rr, "\rho^*" description] &  & \mathbb{O}\mathbbm{pt}(\ce) \arrow[dd, "\mathsf{U}" description] \\
                                                                        &  &                                                                  \\
                                                                        &  & \P({\mathbbm{pt}(\ce)})                                        
  \end{tikzcd}
  \end{center}
  It is easy to show that $\rho^*$ is cocontinuous and preserves finite limits and is thus the inverse image functor of a geometric morphism $\rho: \mathbb{O}\mathbbm{pt}(\ce) \to \ce$ as desired.

  \item[($\lambda$)] Recall that a morphism of ionads $\lambda: \cx \to \mathbbm{pt}\mathbb{O} \cx $ is the data of a functor $\lambda: X \to \mathsf{pt}\mathbb{O}\cx$ together with a lifting $\lambda^\sharp: \mathbb{O}\cx\to \mathbb{O}\mathsf{pt}\mathbb{O}\cx.$ We only provide  $\lambda: X \to \mathsf{pt}\mathbb{O}\cx$, $\lambda^{\sharp}$ is induced by \Cref{morphism of ionads generator}. Indeed such a functor is the same of a functor $\lambda: X \to \text{Cocontlex}(\mathbb{O}\cx, \Set)$. Define, $$\lambda(x)(s) = (\mathsf{U}(s))(x).$$ From a more conceptual point of view, $\lambda$ is just given by the composition of the  functors, $$X \stackrel{\mathsf{eval}}{\longrightarrow} \text{Cocontlex}(\P(X), \Set) \stackrel{- \circ \mathsf{U}}{\longrightarrow} \text{Cocontlex}(\mathbb{O}\cx, \Set).$$
  \end{itemize}
To finish the proof we now show that the induced functors in the diamond below induce an equivalence of categories between \textit{left} and \textit{right}.

\[\begin{tikzcd}[scale=0.5]
  & {\text{Topoi}(\mathbb{O}\cx, \mathbb{O}\mathbbm{pt}\ce)} \\
  {\text{BIon}(\cx, \mathbbm{pt}\ce)} && {\text{Topoi}(\mathbb{O}\cx, \ce)} \\
  & {\text{BIon}(\mathbbm{pt} \mathbb{O}\cx, \mathbbm{pt}\ce)}
  \arrow["{\mathbb{O}}"{description}, from=2-1, to=1-2]
  \arrow["{\rho \circ (-)}"{description}, from=1-2, to=2-3]
  \arrow["{\mathbbm{pt}}"{description}, from=2-3, to=3-2]
  \arrow["{(-) \circ \lambda}"{description}, from=3-2, to=2-1]
\end{tikzcd}\]

The proof of this is very involved but not particularly instructive, thus we will only show that one of the two composition is naturally isomorphic to the identity. To do so, let $g \in{\text{Topoi}(\mathbb{O}\cx, \ce)}$. We will show that when we proceed clockwise in the diagram above we land where we started from.  $\mathbbm{pt}(g) \circ \lambda$ is described by the diagram below.

\[\begin{tikzcd}
  {\mathbb{O}\mathbbm{pt} \ce} && {\mathbb{O}\mathbbm{pt}\mathbb{O} \cx)} && {\mathbb{O}(\cx)} \\
  && {} \\
  {\P(\mathsf{pt} \ce)} && {\P(\mathsf{pt}\mathbb{O} \cx)} && {\P (\cx)}
  \arrow["{\mathsf{pt}(g)^*}"{description}, from=3-1, to=3-3]
  \arrow["{\lambda^*}"{description}, from=3-3, to=3-5]
  \arrow["{\mathsf{U}_{\mathbbm{pt} \ce}}"{description}, from=1-1, to=3-1]
  \arrow["{\mathsf{pt}(g)^\sharp}"{description}, from=1-1, to=1-3]
  \arrow["{\lambda^\sharp}"{description}, from=1-3, to=1-5]
  \arrow["{\mathsf{U}_\cx}"{description}, from=1-5, to=3-5]
  \arrow["{\mathsf{U}_{\mathbbm{pt}\mathbb{O} \cx}}"{description}, from=1-3, to=3-3]
\end{tikzcd}\]

Now notice that $\rho \circ \mathbb{O}(\mathbbm{pt}(g) \circ \lambda)$ selects the first arrow of the diagram above (this is the action of $\mathbb{O}$) and the composes it with $\rho^*$, as shown below.
\[\begin{tikzcd}
  \ce && {\mathbb{O}\mathbbm{pt} \ce} && {\mathbb{O}\mathbbm{pt}\mathbb{O} \cx)} && {\mathbb{O}(\cx)} \\
  &&&& {} \\
  && {\P(\mathsf{pt} \ce)} && {\P(\mathsf{pt}\mathbb{O} \cx)} && {\P (\cx)}
  \arrow["{\mathsf{pt}(g)^*}"{description}, from=3-3, to=3-5]
  \arrow["{\lambda^*}"{description}, from=3-5, to=3-7]
  \arrow["{\mathsf{U}_{\mathbbm{pt} \ce}}"{description}, from=1-3, to=3-3]
  \arrow["{\mathsf{pt}(g)^\sharp}"{description}, from=1-3, to=1-5]
  \arrow["{\lambda^\sharp}"{description}, from=1-5, to=1-7]
  \arrow["{\mathsf{U}_\cx}"{description}, from=1-7, to=3-7]
  \arrow["{\mathsf{U}_{\mathbbm{pt}\mathbb{O} \cx}}"{description}, from=1-5, to=3-5]
  \arrow["{\rho^*}"{description}, from=1-1, to=1-3]
  \arrow["{ev^*_\ce}"{description}, from=1-1, to=3-3]
\end{tikzcd}\]

To finish the proof it is thus enough to show that the composition $\lambda^\sharp \circ \mathsf{pt}(g)^\sharp \circ \rho^*$ coincides with the original $g^*$ us to isomorphism. Since $\mathsf{U}_\cx$ is comonadic (thus conservative and full on isomorphisms), we can check it on the second row of the diagram above). It is thus enough to show that $\lambda^* \circ \mathsf{pt}(g)^* \circ ev^*_\ce$ coincides with the precomposition $g^*$. The latter is a tautology.
\end{proof}

\subsection{Interaction} \label{interaction}

In this short subsection we shall convince the reader that the posetal version of the Scott-Isbell story \textit{embeds} in the categorical one.

\begin{center}
\begin{tikzcd}[scale=0.5]
                              & \text{Loc} \arrow[dd] \arrow[ld] \arrow[rr, "\mathsf{Sh}" description, dashed] &             & \text{Topoi} \arrow[dd] \arrow[ld] \\
\text{Top} \arrow[rr, dashed] &                                                                                & \text{BIon} &                                    \\
                              & \text{Pos}_\omega \arrow[lu] \arrow[rr, "i" description, dashed]           &             & \text{Acc}_\omega \arrow[lu]      
\end{tikzcd}
\end{center}
We have no applications for this observation, thus we do not provide all the details that would require the description of an enormous amount of functors relating all the categories that we have mentioned. Yet, we show the easiest aspect of this phenomenon. Let us introduce and describe the following diagram,

\begin{center}

\begin{tikzcd}
\text{Loc} \arrow[rr, "\mathsf{Sh}" description] \arrow[dd, "\mathsf{pt}" description] &  & \text{Topoi} \arrow[dd, "\mathsf{pt}" description] \\
                                                                                       &  &                                                    \\
\text{Pos}_\omega \arrow[rr, "i" description]                                   &  & \text{Acc}_\omega                                 
\end{tikzcd}
\end{center}

\begin{rem}[$\mathsf{Sh}$ and $i$]
\begin{itemize}
  \item[]
  \item[$\mathsf{Sh}$] It is well known that the sheafification functor $\mathsf{Sh}: \text{Loc} \to \text{Topoi}$ establishes \text{Loc} as a full subcategory of $\text{Topoi}$ in a sense made precise in \cite{sheavesingeometry}[IX.5, Prop. 2 and 3].
  \item[$i$] This is very easy to describe. Indeed any poset with directed suprema is an accessible category with directed colimits and a function preserving directed suprema is precisely a functor preserving directed colimits.
  \end{itemize}
\end{rem}

\begin{prop} \label{recovertopology}
The diagram above commutes.
\end{prop}
\begin{proof}
This is more or less tautological from the point of view of \cite[IX.5, Prop. 2 and 3]{sheavesingeometry}. In fact, $$\mathsf{pt}(L)=\text{Loc}(\mathbb{T}, L) \cong \text{Topoi}(\mathsf{Sh}(\mathbb{T}), \mathsf{Sh}(L)) \cong \mathsf{pt}(\mathsf{Sh}(L)).$$ 
\end{proof}

 

 \section{Sober ionads and topoi with enough points} \label{soberenough}

In this section we show that the categorified Isbell adjunction is idempotent, providing a categorification of  \Cref{topologysoberandspatial}. The notion of sober ionad is a bit unsatisfactory and lacks an intrinsic description. Topoi with enough points have been studied very much in the literature. Let us give (or recall) the two definitions.

\begin{defn}[Sober ionad]
A ionad is sober if $\lambda$ is an equivalence of ionads.
\end{defn}

\begin{defn}[Topos with enough points]
A topos has enough points if the inverse image functors from all of its points are jointly conservative.
\end{defn}

 \begin{thm}[Idempotency of the categorified Isbell duality] \label{idempotencyisbellcategorified}
 The following are equivalent:
 \begin{enumerate}
  \item $\ce$ has enough points;
  \item $\rho: \mathbb{O}\mathbbm{pt}(\ce) \to \ce$ is an equivalence of categories;
  \item $\ce$ is equivalent to a topos of the form $\mathbb{O}(\cx)$ for some bounded ionad $\cx$.
 \end{enumerate}
 \end{thm}
 \begin{proof}
  \begin{enumerate}
    \item[] 
  \item[$(1) \Rightarrow (2)$] Going back to the definition of $\rho$ in \Cref{categorified isbell adj thm}, it's enough to show that $ev^*$ is comonadic. Since it preserves finite limits, it's enough to show that it is conservative to apply Beck's (co)monadicity theorem. Yet, that is just a reformulation of having \textit{enough points}.
  \item[$(2) \Rightarrow (3)$] Trivial.
  \item[$(3) \Rightarrow (1)$] \cite[Rem. 2.5]{ionads}.
 \end{enumerate}
 \end{proof}

\begin{thm}
 The following are equivalent:
 \begin{enumerate}
  \item $\cx$ is sober;
  \item $\cx$ is of the form $\mathbbm{pt}(\ce)$ for some topos $\ce$.
 \end{enumerate}
 \end{thm}
 \begin{proof}
 For any adjunction, it is enough to show that either the monad or the comonad is idempotent\footnote{Of course, since this is a $2$-adjunction, by idempotent we mean that the (co)multiplication is an \textit{equivalence} of categories.}, to obtain the same result for the other one.
 \end{proof}

 \subsection{From Isbell to Scott: categorified} \label{isbelltoscottcategorified}
 
 This section is a categorification of its analog  \Cref{topologyscottfromisbell} and shows how to infer results about the tightness of the Scott adjunction from the Isbell adjunction. We mentioned that there exists a natural transformation as described by the diagram below.

 \begin{center}
\begin{tikzcd}
           & \text{Topoi} \arrow[lddd, "\mathbbm{pt}" description] \arrow[rddd, "\mathsf{pt}"{name=pt}, description] &                                               \\
           &                                                                                             &                                               \\
           &                                                                                             &                                               \\
\text{BIon} \ar[Rightarrow, from=pt, shorten >=1pc, "\iota", shorten <=1pc]&                                                                                             & \text{Acc}_{\omega} \arrow[ll, "\mathsf{ST}"]
\end{tikzcd}
\end{center}

\begin{defn}[$\iota$]
Let $\epsilon : \mathsf{Spt} \Rightarrow 1$ be the counit of the Scott adjunction and call $\phi : \mathbb{O} \circ \mathsf{ST} \Rightarrow \mathsf{S}$ the isomorphism appearing in \Cref{categorifiedscotttopology}. Then $\iota$ is the mate of the (pseudo)-natural transformation below along the Isbell adjunction \[\mathbb{O} \circ \mathsf{ST} \circ \mathsf{pt} \stackrel{\phi_\mathsf{pt}}{\Longrightarrow}  \mathsf{Spt} \stackrel{\epsilon}{\Longrightarrow} 1.\] 
\end{defn}

\begin{rem}
Let us provide a more down-to-earth description of $\iota$. Spelling out the content of the diagram above, $\iota$ should be a morphism of ionads $$\iota:  \mathsf{ST} \mathsf{pt} (\ce) \to \mathbbm{pt}(\ce). $$
Recall that the underling category of these two ionads is $ \mathsf{pt} (\ce)$ in both cases.

We define $\iota:  \mathsf{ST} \mathsf{pt} (\ce) \to \mathbbm{pt}(\ce)$ to be the identity on the underlying categories, $\iota = (1_{\mathbbm{pt}(\ce)}, \iota^\sharp)$ while $\iota^\sharp$ is induced by the following assignment defined on the basis of the ionad $\ce \to \mathsf{Spt}(\ce)$, $$\iota^\sharp(x)(p) = p^*(x). $$

The reader might have noticed that $\iota^\sharp$ is precisely the counit of the Scott adjunction.
\end{rem}

\begin{thm}[From Isbell to Scott, cheap version] \label{cheapidempotencyscott} The following are equivalent:
\begin{enumerate}
\item $\ce$ has enough points and $\iota$ is an equivalence of ionads.
\item The counit of the Scott adjunction is an equivalence of categories.
\end{enumerate}
\end{thm}
\begin{proof}
This is completely obvious from the previous discussion.
\end{proof}

This result is quite disappointing in practice and we cannot accept it as it is. Yet, having understood that the Scott adjunction is not  the same as Isbell one was very important conceptual progress in order to guess the correct statement for the Scott adjunction.

\begin{rem}[The Scott adjunction is not a biequivalence: Fields] \label{fields}
In this remark we provide a topos $\cf$ such that the counit \[\epsilon_\cf : \mathsf{Spt} \cf \to \cf\] is not an equivalence of categories. Let $\cf$ be the classifying topos of the theory of geometric fields \cite{elephant2}[D3.1.11(b)]. Its category of points is the category of fields $\mathsf{Fld}$, since this category is finitely accessible the Scott topos $\mathsf{Spt}(\cf)$ is of presheaf type, \[\mathsf{Spt}(\cf) = \mathsf{S} (\mathsf{Fld} )      \simeq     \Set^{\mathsf{Fld}_\omega}.\]
It was shown in \cite{10.2307/30041767}[Cor 2.2] that $\cf$ cannot be of presheaf type, and thus $\epsilon_\cf$ cannot be an equivalence of categories.

\end{rem}

In the next section we provide a more satisfactory version of \Cref{cheapidempotencyscott}. 

 \subsection{Covers}
 In order to provide a more satisfying version of \Cref{cheapidempotencyscott}, we need to introduce a tiny bit of technology, namely finitely accessible covers. For an accessible category with directed colimits $\ca$ a (finitely accessible) \textit{cover} $$\cl : \mathsf{Ind}(C) \to \ca $$ will be one of a class of cocontinuous (pseudo)epimorphisms (in Cat) having many good properties. They will be helpful for us in the discussion. Covers were introduced for the first time in \cite[4.5]{aec} and later used in \cite[2.5]{LB2014}.

 \begin{rem}[Generating covers] 
Let $\ca$ be $\kappa$-accessible and let us focus on the following diagram.

\[\begin{tikzcd}
  & {\ca_\kappa} && \ca_\kappa \\
  {\mathsf{Ind}\ca_\kappa} & \ca && {\mathsf{Ind}\ca_\kappa} & \ca
  \arrow["\iota", from=1-2, to=2-2]
  \arrow["\alpha"', from=1-2, to=2-1]
  \arrow["{\ca(\iota-,-)}", dashed, from=2-2, to=2-1]
  \arrow["\iota", from=1-4, to=2-5]
  \arrow["\alpha"', from=1-4, to=2-4]
  \arrow["{\lan_\alpha \iota}"', dotted, from=2-4, to=2-5]
\end{tikzcd}\]


Let us proceed to a description of the dashed and the dotted arrow.
\begin{itemize}
  \item[$\ca(\iota-,-)$] also known as \textit{the nerve of f}, maps $a \mapsto \ca(\iota-,a)$. $\ca(\iota-,a) : \ca_\kappa^\circ \to \Set$ is $\kappa$-flat, because $\ca$ is $\kappa$-accessible, thus it is flat (being a weaker condistion). Thus $\ca(\iota-,-)$ lands in $\text{Flat}(\ca_\kappa^\circ, \Set) = \mathsf{Ind}(\ca_\kappa)$. Moreover, $\ca(\iota-,-)$ is fully faithful because $ \ca_\kappa$ is a (dense) generator in $\ca$ and preserve $\kappa$-directed colimits. This functor does not lie in $\text{Acc}_{\omega}$ in general.
  \item[$\lan_{\alpha}(\iota)$] exists by the universal property of of the $\mathsf{Ind}$-completion, indeed $\ca$ has directed colimits by definition. For a concrete perspective $\lan_{\alpha}(\iota)$ is evaluating a formal directed colimit on the actual directed colimit in $\ca$. 
  \end{itemize}

  These two functors yield an adjunction \[\lan_{\alpha}(\iota) \dashv \ca(\iota-,-),\] that establishes $\ca$ as a reflective embedded subcategory of $\mathsf{Ind} \ca_\kappa$. Since $\lan_{\alpha}(\iota)$ is a left adjoint, it lies in $\text{Acc}_{\omega}$.
 \end{rem}

\begin{defn}\label{cover} When $\ca$ is a $\kappa$-accessible category we will indicate with $\mathcal{L}_{\ca}^{\kappa}$ the functor  $\lan_{\alpha}(\iota)$ in the previous remark and we call it a \textit{cover} of $\ca$. \[\mathcal{L}_{\ca}^{\kappa} : \mathsf{Ind} \ca_\kappa\to \ca .\]
\end{defn}

Notice that  every object $\ca$ in $\text{Acc}_{\omega}$ has a (proper class of) finitely accessible covers.
 
 \begin{notation}
 When it's evident from the context, or it is not relevant at all we will not specify the cardinal on the top, thus we will just write $\mathcal{L}_{\ca}$ instead of $\mathcal{L}^{\kappa}_{\ca}$.
 \end{notation}
 
 \begin{rem} When $\lambda \geq \kappa$, $\mathcal{L}_{\ca}^{\lambda}\cong  \mathcal{L}_{\ca}^{\kappa} \mathcal{L}_{\kappa}^{\lambda}$ for some transition map $\mathcal{L}_{\kappa}^{\lambda}$. We did not find any application for this, thus we decided not to go in the details of this construction.
 \end{rem}
  


\begin{rem} This construction appeared for the first time in \cite[4.5]{aec} and later in \cite[2.5]{LB2014}, where it is presented as the analogue of Shelah's presentation theorem for AECs \cite[2.6]{LB2014}. The reader that is not familiar with formal category theory might find the original presentation more down to earth. In \cite[2.5]{LB2014} it is also shown that under interesting circumstances the cover is faithful.
\end{rem}

\subsection{On the (non) idempotency of the Scott adjunction}\label{scottidempotency}

This subsection provides a better version of \Cref{cheapidempotencyscott}. It is based on a technical notion (topological embedding) that we have defined and studied in \cite{thcat}. Let us recall the definition, we refer to that paper for the relevant results.

\begin{defn}[Topological embedding]
Let $f: \ca \to \cb$ be a $1$-cell in $\text{Acc}_{\omega}$. We say that $f$ is a topological embedding if $\mathsf{S}(f)$ is a geometric embedding.
\end{defn}

\begin{thm}  \label{enoughpoints}
A Scott topos $\cg \cong \mathsf{S}(\ca)$ has enough points.
\end{thm}
\begin{proof}
It is enough to show that $\cg$ admits a geometric surjection from a presheaf topos \cite[2.2.12]{elephant2}. Let $\kappa$ be a cardinal such that $\ca$ is $\kappa$-accessible. We claim that the $1$-cell \[\mathsf{S}(\mathcal{L}^{\kappa}_{\ca}): \mathsf{S}\mathsf{Ind } \ca_\kappa \to \mathsf{S} \ca,\] described in \Cref{cover} is the desired geometric surjection for the Scott topos $\mathsf{S}(\ca)$. Indeed $\mathsf{S}\mathsf{Ind } \ca_\kappa$ is  a presheaf topos, as \[\mathsf{S}\mathsf{Ind } \ca_\kappa = \text{Acc}_\omega(\mathsf{Ind } \ca_\kappa, \Set) \simeq \Set^{\ca_\kappa}.\] 
We now have to show that $\mathsf{S}(\mathcal{L}^{\kappa}_{\ca})$ is a geometric surjection. Among the possible characterizations, we show that the inverse image functor. $\mathsf{S}(\mathcal{L}^{\kappa}_{\ca})^*$ is faithful. Now, $\mathsf{S}(\mathcal{L}^{\kappa}_{\ca})^*$ is precisely the precomposition $g \mapsto g \circ \mathcal{L}^{\kappa}_{\ca}$. It was shown in \cite[4.6]{laxepi} that such functor is faithful (at the level of presheaf categories and thus a fortiori on the Scott topos) if $\mathcal{L}^{\kappa}_{\ca}$ is a pseudo epimorphism in Cat. Indeed this is true, because it has even a section in Cat, namely $\ca(\iota-,-)$.
\end{proof}

\begin{thm} The following are equivalent. \label{scottidempotency}
\begin{enumerate} 
    \item The counit $\epsilon: \mathsf{Spt}(\ce) \to \ce$ is an equivalence of categories.
    \item  $\ce$ has enough points and for all presentations $i^*: \Set^X \leftrightarrows \ce: i_*$, $\mathsf{pt}(i)$ is a topological embedding.
   \item $\ce$ has enough points and there exists a presentation $i^*: \Set^X \leftrightarrows \ce: i_*$ such that $\mathsf{pt}(i)$ is a topological embedding.
  
\end{enumerate} 
\end{thm}
\begin{proof}
As expected, we follow the proof strategy hinted at by the enumeration.
\begin{itemize}

    \item[$1) \Rightarrow 2)$] By \Cref{enoughpoints}, $\ce$ has enough points. We only need to show that for all presentations $i^*: \Set^X \leftrightarrows \ce: i_*$, $\mathsf{pt}(i)$ is a topological embedding. In order to do so, consider the following diagram,

    \begin{center}
    \begin{tikzcd}
\mathsf{Spt}\ce \arrow[dd, "\epsilon_\ce" description] \arrow[rr, "\mathsf{Spt}(i)" description] &  & \mathsf{Spt}\Set^X \arrow[dd, "\epsilon_{\Set^X}" description] \\
                                                                                                 &  &                                                                \\
\ce \arrow[rr, "i" description]                                                                  &  & \Set^X                                                        
\end{tikzcd}
    \end{center}
    By \cite{thcat}[Rem. 2.13] and the hypotheses of the theorem, one obtains that $\mathsf{Spt}(i)$ is naturally isomorphic to a composition of geometric embedding\[\mathsf{Spt}(i) \cong \epsilon_{\Set^X}^{-1} \ \circ \  i  \ \circ \  \epsilon_\ce,\]  and thus is a geometric embedding. This shows precisely that $\mathsf{pt}(i)$ is a topological embedding.

        \item[$2) \Rightarrow 3)$] Obvious.
  \item[$3) \Rightarrow 1)$] It is enough to prove that $\epsilon_\ce$ is both a surjection and a geometric embedding of topoi. $\epsilon_\ce$ is a surjection, indeed since $\ce$ has enough points, there exist a surjection $q: \Set^X \twoheadrightarrow \ce$, now we apply the comonad $\mathsf{S}\mathsf{pt}$ and we look at the following diagram,

    \begin{center}
  \begin{tikzcd}
\mathsf{S}\mathsf{pt}\Set^X \arrow[dd, "\mathsf{S}\mathsf{pt}(q)" description] \arrow[rr, "\epsilon_{\Set^X}" description] &  & \Set^X \arrow[dd, "q" description, two heads] \\
                                                                                                                           &  &                                               \\
\mathsf{S}\mathsf{pt}\ce \arrow[rr, "\epsilon_\ce" description]                                                            &  & \ce                                          
\end{tikzcd}
  \end{center}
  Now, the counit arrow on the top is an isomorphism, because $\Set^X$ is a presheaf topos. Thus $\epsilon_{\ce} \circ (\mathsf{S}\mathsf{pt})(q) $ is (essentially) a factorization of $q$. Since $q$ is a geometric surjection so must be $\epsilon_\ce$.  In order to show that $\epsilon_\ce$ is a geometric embedding, we use again the following diagram over the given presentation $i$.

  \begin{center}
  \begin{center}
    \begin{tikzcd}
\mathsf{Spt}\ce \arrow[dd, "\epsilon_\ce" description] \arrow[rr, "\mathsf{Spt}(i)" description] &  & \mathsf{Spt}\Set^X \arrow[dd, "\epsilon_{\Set^X}" description] \\
                                                                                                 &  &                                                                \\
\ce \arrow[rr, "i" description]                                                                  &  & \Set^X                                                        
\end{tikzcd}
    \end{center}
  \end{center}
This time we know that $\mathsf{Spt}(i)$ and $i$ are geometric embeddings, and thus $\epsilon_\ce$ has to be so.
\end{itemize}
\end{proof}

\begin{rem}
The version above might look like a technical but not very useful improvement of \Cref{cheapidempotencyscott}. Instead, in the following Corollary we prove a non-trivial result based on a characterization (partial but useful) of topological embeddings contained in the Toolbox of \cite{thcat}.
\end{rem}

\begin{exa} \label{enoughpointscocomplete}
Let $\ce$ be a topos with enough points, together with a presentation $i: \ce \to \Set^C$. If  $\mathsf{pt}(\ce)$ is reflective via $\mathsf{pt}(i)$, then $\epsilon_\ce$ is an equivalence of categories.
\end{exa}
\begin{proof}
We verify the condition (3) of the previous theorem. By \cite{thcat}[Thm. 4.10], $\mathsf{pt}(i)$ must be a topological embedding.
\end{proof}

\begin{exa}[The Schanuel topos]
The Schanuel topos $\mathcal{S}$ is the topos of sheaves over the category $\text{Fin}_\text{m}$ of finite sets and monomorphisms equipped with the atomic topology \cite[pp. 79f, 691, 925]{Sketches}. Obviously it admits a canonical presentation as a subtopos $i: \mathcal{S} \hookrightarrow \Set^{\text{Fin}^\circ_\text{m}}$. The corresponding map on points is the inclusion of countable sets (and monos) in the category of sets and monos, \[\mathsf{pt}(i) : \Set_{\geq \omega, m} \to \Set_{m}.\]One can see the whole \cite[Sec. 4]{simon} as a proof of the fact that such a map is a topological embedding and thus the Schanuel topos can be recovered from its category of points. Along these lines we have other similar examples provided in the last section of \cite{simon}.
\end{exa}

Notice that the example of the Schanuel topos displays the complexity of the situation we are in. Using \Cref{scottidempotency} it is still quite hard to show that $\epsilon$ is an equivalence of categories (\cite[Sec. 4]{simon} is far from being trivial), because we lack a satisfactory theory of topological embeddings, but it makes it at least feasible.

\section*{Acknowledgements}
I am indebted to my advisor, \textit{Jiří Rosický}, for the freedom and the trust he blessed me with during the years of my Ph.D, not to mention his sharp and remarkably blunt wisdom. I am grateful to \textit{Axel Osmond} for some very constructive discussions on this topic, to \textit{Richard Garner} for several comments that have improved the paper both technically and linguistically, to \textit{Tomáš Jakl} who provided all the references for the history of the theory of locales. I am grateful to \textit{Andrea Gagna} for having read and commented a preliminary draft of this paper. Finally, I am grateful to the \textit{anonymous referee} for their suggestions and comments.

\bibliography{thebib}
\bibliographystyle{alpha}

\end{document}